\def\timestamp{%
Time-stamp: <elemChains-submit.tex: Thursday 08-09-2011 at 16:36:16 (cest)>}
\def\stripname Time-stamp: <#1 #2>{#2}
\edef\filedate{\expandafter\stripname\timestamp}

\documentclass{amsart}

\newcommand\orpr[2]{\langle#1,#2\rangle}

\DeclareMathSymbol\restr2{AMSa}{"16} 

\newcommand\axiom{\mathsf}
\newcommand\CH{\axiom{CH}}
\newcommand\OCA{\axiom{OCA}}
\newcommand\PFA{\axiom{PFA}}
\newcommand\ZFC{\axiom{ZFC}}

\newcommand\pr{\operatorname{pr}}
\newcommand\cl{\operatorname{cl}}
\newcommand\Int{\operatorname{int}}

\theoremstyle{plain}
\newtheorem{theorem}{Theorem}[section]
\newtheorem{lemma}[theorem]{Lemma}
\newtheorem{proposition}[theorem]{Proposition}
\newtheorem{corollary}[theorem]{Corollary}

\theoremstyle{definition}
\newtheorem{definition}[theorem]{Definition}
\newtheorem*{claim}{Claim}

\theoremstyle{remark}

\newtheorem{question}{Question}[section]

\numberwithin{equation}{section}

\usepackage{amsrefs}

\begin{document}
\title{Elementary chains and  compact spaces with a small diagonal}

\author[Alan Dow]{Alan Dow\dag}
\address{
Department of Mathematics\\
UNC-Charlotte\\
9201 University City Blvd. \\
Charlotte, NC 28223-0001}
\email{adow@uncc.edu}
\urladdr{http://www.math.uncc.edu/\~{}adow}
\thanks{\dag Research of the first author was supported by NSF grant 
        No.\ NSF-DMS-0901168.} 

\author{Klaas Pieter Hart}
\address{Faculty of Electrical Engineering, Mathematics and Computer Science\\
         TU Delft\\
         Postbus 5031\\
         2600~GA {} Delft\\
         the Netherlands}
\email{k.p.hart@tudelft.nl}
\urladdr{http://fa.its.tudelft.nl/\~{}hart}



\date{\filedate}

\begin{abstract}
It is a well known open problem if, in $\ZFC$, each compact space with a small 
diagonal is metrizable. 
We explore properties of compact spaces with a small diagonal using 
elementary chains of submodels. 
We prove that ccc subspaces of such spaces have countable $\pi$-weight. 
We generalize a result of Gruenhage about spaces which are metrizably fibered.
Finally we discover that if there is a Luzin set of reals, then every compact 
space with a small diagonal will have many points of countable character. 
\end{abstract}

\maketitle


\section*{Introduction}

In~\cite{Husek} Hu\v{s}ek defined a space, $X$, to have an 
\emph{$\omega_1$-accessible diagonal} if there is a $\omega_1$-sequence
$\bigl\langle\orpr{x_\alpha}{y_\alpha}:\alpha<\omega_1\bigr\rangle$
in~$X^2$ that converges to the diagonal~$\Delta(X)$ in the sense that
every neighbourhood of the diagonal contains a tail of the sequence.
Hu\v{s}ek also mentions that Van Douwen referred to spaces that do 
\emph{not} have an $\omega_1$-accessible diagonal as 
\emph{having a small diagonal}.
The latter has gained more currency since Zhou's~\cite{ZhouDiag} and 
the definition has been cast in a more positive form:
~$X$~has a small diagonal if every uncountable subset of~$X^2$ that is disjoint
from the diagonal has an uncountable subset whose closure is disjoint from the
diagonal. 
For brevity's sake we will say that a space is csD if it is a \textbf{c}ompact 
Hausdorff space with a \textbf{s}mall \textbf{D}iagonal. 

There are a number of very interesting results known for csD
spaces and we recommend \cite{GaryDiag} as an excellent reference. 
In particular, it is known that csD spaces have countable tightness
 (\cite{JuSzCvgt}) and 
that it follows from $\CH$ that csD spaces are metrizable 
(\cites{Husek,JuSzCvgt}).
One of our main results is that ccc subspaces of a csD space have countable 
$\pi$-weight. 
In~\cite{Tkachuk94} Tkachuk describes a space as \emph{metrizably fibered} 
if there is a continuous map onto a metric space with the property that 
each point preimage (fiber) is also metrizable. 
Let us  say that a space is \emph{weight $\kappa$ fibered} if the obvious 
generalization is satisfied: there is a map onto a space with weight at 
most~$\kappa$ so that each fiber also has weight at most~$\kappa$. 
In~\cite{GaryDiag} Gruenhage showed that a metrizably fibered csD space is 
metrizable. 
We will show that this is also true for weight $\omega_1$ fibered spaces. 

Though the main question on csD spaces is whether they are metrizable it is 
at present not even known if they must have points of countable character. 
We uncover the surprising connection that if there is a Luzin set of reals, 
then each csD space does have points of countable character.

\section{Preliminaries}

We begin by citing a convenient characterization of csD spaces obtained
by Gruenhage in~\cite{GaryDiag}.
We say that a sequence 
$\bigl\langle\orpr{x_\alpha}{y_\alpha}:\alpha\in \omega_1\bigr\rangle$ 
of pairs is \emph{$\omega_1$-separated} if there is an uncountable subset~$A$
of~$\omega_1$ such that 
$\{x_\alpha:\alpha\in A\}$ and $\{y_\alpha:\alpha\in A\}$ have disjoint closures.
Gruenhage showed that a compact space is csD if and only if every uncountable 
sequence of pairs is $\omega_1$-separated.

\subsection{Elementary sequences}

The key to uncovering why a non-metric compact space might not be csD is to 
select the right $\omega_1$-sequence of pairs. 
We will explore a method of using chains of countable elementary 
submodels for this purpose. 

For a cardinal~$\theta$ we let $H(\theta)$ denote the collection of all sets 
whose transitive closure has cardinality less than~$\theta$ 
(see~\cite{Kunen80}*{Chapter~IV}).   
An $\omega_1$-sequence $\langle M_\alpha:\alpha<\omega_1\rangle$ of countable 
elementary substructures of~$H(\theta)$ that satisfies 
$\langle M_\beta:\beta\le\alpha\rangle\in M_{\alpha+1}$ for all~$\alpha$
and $M_\alpha=\bigcup_{\beta<\alpha}M_\beta$ for all limit~$\alpha$
will simply be called an \emph{elementary sequence}. 

It will be convenient to assume that the spaces considered in this paper are 
subspaces of~$[0,1]^\kappa$ for some suitable cardinal~$\kappa$
(usually the weight of the space under consideration).
By a basic open subset of~$[0,1]^\kappa$ we mean a set that is a product 
$\prod_{\xi\in\kappa} I_\xi$ where each~$I_\xi$ is a relatively open subinterval 
of~$[0,1]$ with rational endpoints such that the set of~$\xi\in\kappa$
for which $I_\xi\neq[0,1]$ --- its \emph{support} --- is finite. 
When a space $X$ is a subspace of $[0,1]^\kappa$ we will use the intersections
of these basic open sets with~$X$ as a base for~$X$.

\subsubsection*{Elementary sequences for a space} 

Let $X$ be a compact space and assume that $X$~is a subset 
of~$[0,1]^\kappa$ where $\kappa$~is the weight of~$X$. 
An \emph{elementary sequence for~$X$} will be an elementary sequence
$\langle M_\alpha:\alpha<\omega_1\rangle$ in~$H\bigl((2^\kappa)^+\bigr)$
such that $X\in M_0$.
  
For each $\alpha\in \omega_1$ we associate two spaces with $X$ and $M_\alpha$:
the first is the closure of $X\cap M_\alpha$, which we denote~$X_\alpha$.
The second is the image of~$X$ under the projection~$\pr_{M_\alpha}$ 
from~$[0,1]^\kappa$ onto~$[0,1]^{M_\alpha\cap \kappa}$:
we write $X_{M_\alpha}=\pr_{M_\alpha}[X]$.

For each $x\in X$ and $\alpha\in \omega_1$ we often denote $\pr_{M_\alpha}(x)$
by $x\restr M_\alpha$ and we write
$[x\restr M_\alpha]=\{y\in X:\pr_{M_\alpha}(y)=\pr_{M_\alpha}(x)\}$.  

We will apply Gruenhage's criterion to sequences of pairs associated to
elementary sequences.

\begin{definition} 
An elementary $\omega_1$-sequence of pairs for a space $X$ is a sequence 
$\bigl<\orpr{x_\alpha}{y_\alpha}:\alpha\in\omega_1\bigr>$ of pairs of points
from~$X$ for which there is some elementary sequence 
$\langle M_\alpha:\alpha\in\omega_1\rangle$ for~$X$
so that $\{x_\alpha,y_\alpha\}\in M_{\alpha+1}$, \ $x_\alpha\neq y_\alpha$ and 
$x_\alpha\restr M_\alpha = y_\alpha \restr M_\alpha$ 
for each $\alpha\in \omega_1$.
\end{definition}

The (seemingly narrow) gap between metrizability and being csD in the class
of compact spaces is revealed in the next two propositions.

\begin{proposition}
A compact space is metrizable if and only if it has no elementary
 $\omega_1$-sequence of pairs. 
\end{proposition}

It is somewhat easier to prove the contrapositive form: a compact space
has uncountable weight iff it has an elementary $\omega_1$-sequence of pairs.

\begin{proposition} 
A compact space $X$ is not csD if and only if it has an elementary 
$\omega_1$-sequence of pairs that is not $\omega_1$-separated.
\end{proposition}

\begin{proof} 
We will actually prove a stronger statement: each $\omega_1$-sequence of 
pairs that is not $\omega_1$-separated contains an elementary 
$\omega_1$-sequence. 
Suppose that $\bigl<\orpr{x_\alpha}{y_\alpha}:\alpha\in\omega_1\bigr>$ is an 
uncountable set of pairs of $X$ that is not $\omega_1$-separated.  
Choose any elementary sequence $\langle M_\gamma:\gamma\in\omega_1\rangle$ 
for~$X$ such that $\bigl<\orpr{x_\alpha}{y_\alpha}:\alpha\in\omega_1\bigr>$ 
is an element of~$M_0$.  
For each $\gamma\in \omega_1$, choose  $\alpha_\gamma$ (if possible) so that 
$x_{\alpha_\gamma}\restr M_\gamma=y_{\alpha_\gamma}\restr M_\gamma$.
By elementarity, if such an $\alpha_\gamma$ exists, it may be chosen 
in~$M_{\gamma+1}$.
Then $\bigl<\orpr{x_{\alpha_\gamma}}{y_{\alpha_\gamma}}:\gamma\in\omega_1\}$
is the desired elementary $\omega_1$-sequence that is not $\omega_1$-separated.

To finish the proof we show that it is always possible to choose 
an~$\alpha_\gamma$.
If not then there will be a~$\gamma$ such that the set~$A$ of~$\alpha$ for
which  $x_\alpha\restr M_\gamma\neq y_\alpha\restr M_\gamma$
is co-countable. 
For each $\alpha\in A$, there are basic open sets $U_\alpha$ and~$W_\alpha$ 
with disjoint closures and supports in $M_\gamma$ such that 
$x_\alpha\in U_\alpha$ and $y_\alpha\in W_\alpha$. 
As $U_\alpha$ and $W_\alpha$ are determined by finite subsets of~$M_\gamma$
they belong to~$M_\gamma$.
As $M_\gamma$~is countable one pair $\orpr{U}{W}$ would be chosen
uncountably often, say for~$\alpha\in B$.
The latter set would witness that our sequence is $\omega_1$-separated.
\end{proof}

\section{Applications of elementary $\omega_1$-sequences}

There will be occasions when one countable elementary substructure will
already do but in our first result a fair amount of care will go into
the construction of an elementary $\omega_1$-sequence of pairs.

\subsection{Cellularity and $\pi$-weight}

We need the notions of $\pi$-bases and local $\pi$-bases.
A $\pi$-base for a space is a collection of non-empty open sets such that
each non-empty open set contains one of them.
A local $\pi$-base for a space at a point is a collection of non-empty open 
sets such that each neighbourhood of the point contains one of them.
Thus one obtains the notions of $\pi$-weight and $\pi$-character:
minimum cardinalities of defining families.

A continuous surjection is said to be irreducible if no proper closed subset 
of the domain maps onto the range or, dually, the image of a set with non-empty
interior has non-empty interior as well.
The latter formulation easily implies that $\pi$-weight is invariant under
irreducible surjections, in both directions.

An easy application of Zorn's Lemma will show that if $f:X\to Y$ is a 
continuous surjection between compact Hausdorff spaces one can find a closed 
subset~$Z$ of~$X$ such that $f\restr Z:Z\to Y$ is irreducible and surjective.

\begin{lemma}\label{piweight} 
If a compact space $X$ has a ccc subspace with uncountable $\pi$-weight, 
then it has a compact ccc subspace  with $\pi$-weight equal to $\omega_1$.
\end{lemma} 

\begin{proof} 
Let $Y$ be a ccc subspace that does not have a countable $\pi$-base. 
Let $K$ be the closure of~$Y$; then $K$~is compact, ccc, and does not have 
a countable $\pi$-base either. 
Let $\langle M_\alpha : \alpha \in \omega_1\rangle$ be an elementary sequence
for $K$ and let $M=\bigcup_{\alpha\in\omega_1}M_\alpha$.  
For each $\alpha\in\omega_1$ there is, by elementarity, a basic open 
subset~$U_\alpha$ of~$K$ that is an element of~$M_{\alpha+1}$ and does not 
contain any non-empty open subset of~$K$ that is a member of~$M_\alpha$.
It follows that $K_M=\pr_M[K]$ has $\pi$-weight $\omega_1$ since 
it has weight at most~$\omega_1$ and does not have countable $\pi$-weight.

Since $K_M$ is a continuous image of $K$, it is ccc. 
Now choose any compact $Z\subset K$ such that $\pr_M\restr Z:Z\to K_M$ 
is irreducible and onto.
Since the map is irreducible, $Z$~is ccc and has the same $\pi$-weight 
(namely~$\omega_1$) as $K_M$. 
\end{proof}

The first main result began as an attempt to see if there was a more direct 
proof, perhaps uncovering more about the nature of csD spaces, that a csD 
space has countable tightness.

\begin{theorem}\label{cccpiwt} 
In a csD space every ccc subspace has countable $\pi$-weight.
\end{theorem}

\begin{proof} 
We proceed by contradiction.
Since a closed subspace of a csD space will also be csD, 
we may apply Lemma~\ref{piweight} and assume that we have a compact
space~$X$ that is csD and ccc, and has $\pi$-weight equal to $\omega_1$.
 
Let $\langle M_\alpha:\alpha\in\omega_1\rangle$ be an elementary sequence
for~$X$ and let $M=\bigcup_{\alpha\in\omega_1}M_\alpha$. 
It follows that the basic open subsets of~$X$ that are members of~$M$ form 
a $\pi$-base for~$X$.
We will define an elementary $\omega_1$-sequence of pairs.

For each $\alpha$ we choose our pair $x_\alpha\neq y_\alpha\in M_{\alpha+1}$ 
so that 
$x_\alpha\restr M_\alpha = y_\alpha\restr M_\alpha$. 
We first choose a non-empty basic open set $U_\alpha$ from $M_{\alpha+1}$ whose 
closure does not contain any non-empty basic open set that is a member 
of~$M_{\alpha}$. 
We may do so since $X$~is assumed to not have a countable $\pi$-base. 
Let $x_\alpha$ be any point in $U_\alpha\cap M_{\alpha+1}$.
Next, we note that $\pr_{M_\alpha}[X\setminus \cl U_\alpha]$ is dense 
in~$X_{M_\alpha}$ so that $\pr_{M_\alpha}[X\setminus U_\alpha]=X_{M_\alpha}$.  
Therefore we can take a closed subset~$Z_\alpha$ of~$X$, disjoint 
from~$U_\alpha$, such that $\pr_{M_\alpha}$ maps $Z_\alpha$ irreducibly 
onto~$X_{M_\alpha}$; 
by elementarity, there is such a~$Z_\alpha$ in~$M_{\alpha+1}$.

\begin{claim} 
There is a point $y_\alpha\in Z_\alpha$ such that for each neighborhood~$W$ 
of~$y_\alpha$, the 
set~$\pr_{M_\alpha}^{-1}\bigl[X_{M_\alpha}\setminus 
                   \pr_{M_\alpha}[\cl{W}]\bigr]$ 
is not dense in a neighborhood of~$x_\alpha$ --- 
in dual form: $x_\alpha$~belongs to the closure of the interior 
of~$\pr_{M_\alpha}^{-1}\bigl[\pr_{M_\alpha}[\cl{W}]\bigr]$ for every 
neighbourhood of~$y_\alpha$.
By elementarity, such a point~$y_\alpha$ can be chosen to be a member  
of~$M_{\alpha+1}$; 
it is immediate that $y_\alpha\neq x_\alpha$ and 
$x_\alpha\restr M_\alpha=y_\alpha \restr M_\alpha$.

\smallskip\noindent\textit{Proof}:
We prove the dual form: if not then there is a cover of~$Z_\alpha$ by basic 
open sets~$W$ for which 
$x_\alpha\notin\cl\Int \pr_{M_\alpha}^{-1}\bigl[\pr_{M_\alpha}[\cl W]\bigr]$.
Take a finite subcover $\{W_1,\ldots, W_n\}$; then 
$X_{M_\alpha}=\bigcup_{i=1}^n\pr_{M_\alpha}[\cl W_i]$ and so
$X=\bigcup_{i=1}^n\pr_{M_\alpha}^{-1}\bigl[\pr_{M_\alpha}[\cl W_i]\bigr]$.
Now the union of the boundaries of these sets is nowhere dense, 
so that in fact 
$X=\bigcup_{i=1}^n\cl\Int\pr_{M_\alpha}^{-1}\bigl[\pr_{M_\alpha}[\cl W_i]\bigr]$,
a contradiction as $x_\alpha$ was assumed not to belong to that union.
\end{claim}

Now we show that the elementary $\omega_1$-sequence of pairs 
$\bigl\langle\orpr{x_\alpha}{y_\alpha} : \alpha\in \omega_1\bigr\rangle$ 
is not $\omega_1$-separated.  
Suppose that $U$ and $W$ are disjoint open subsets of~$X$ that have disjoint 
closures, and that there is an uncountable subset~$A$ of~$\omega_1$ such 
that $x_\alpha\in U$ and $y_\alpha\in W$ for all $\alpha\in A$.  
Let $\mathcal U$ and $\mathcal W$ denote the families of basic
open subsets of $X$ that are contained in $U$ and $W$ respectively.
Since $X$ is ccc and the family of basic open sets that are members of~$M$
forms a $\pi$-base for~$X$ there is a $\delta<\omega_1$ such that the unions 
of~$\mathcal{U}_\delta=M_\delta\cap \mathcal U$ 
and~$\mathcal{W}_\delta=M_\delta\cap \mathcal W$ are dense in~$U$ and~$W$
respectively.  
Note that each member, $O$, of $\mathcal{U}\cup\mathcal{W}$ has its support
in~$M_\delta$ so that it satisfies 
$O=\pr_{M_\alpha}^{-1}\bigl[\pr_{M_\alpha}[O]\bigr]$.

Let $\alpha\in A$ be larger than $\delta$.
By construction each member of $\mathcal U_\delta$ is disjoint from~$\cl W$
and because its support is in~$M_\alpha$ it is also disjoint from
$\pr_{M_\alpha}^{-1}\bigl[\pr_{M_\alpha}[\cl{W}]\bigr]$.
It follows that $\bigcup \mathcal U_\delta$ is contained in 
$\pr_{M_\alpha}^{-1}\bigl[X_{M_\alpha}
 \setminus\pr_{M_\alpha}[\cl{W}]\bigr]$. 
Since $\bigcup \mathcal U_\alpha$ is
dense in~$U$, this contradicts the conditions in Claim 1.
\end{proof}

We draw attention to the fact that it follows from Theorem~\ref{cccpiwt} 
that a csD space will not map onto $[0,1]^{\omega_1}$, and therefore, 
by \v{S}apirovski\u\i's famous result from~\cite{MR580628} that it will have 
(many) points of countable $\pi$-character.

\subsection{Weight $\omega_1$ fibered}

We will see in Theorem~\ref{t=w} that a similar application of elementary 
sequences will imply that a csD space will have a property stronger than 
countable tightness.
This approach was inspired by the Juh\'asz-Szentmikl\'ossy proof 
from~\cite{JuSzCvgt} where it is shown that if a compact space does not
have countable tightness, then it will contain a converging (free)
$\omega_1$-sequence 
(also making essential use of \v{S}apirovski\u\i's result). 
A csD space can not contain a (co-countably) converging $\omega_1$-sequence.
We will need a strengthening of this result.  
A point~$x$ is commonly called \emph{condensation point} of a set~$A$ if 
every neighborhood of~$x$ contains uncountably many points of~$A$.

\begin{proposition}\label{retract} 
Let $A$ be an uncountable subset of a csD space~$X$ whose set of condensation
points is metrizable, then there is a co-countable subset~$B$ of~$A$ 
with a metrizable closure.
\end{proposition}

\begin{proof} 
Let $\langle M_\alpha:\alpha\in\omega_1\rangle$ be an elementary sequence 
for~$X$ such that $A\in M_0$ and put $M=\bigcup_{\alpha<\omega_1}M_\alpha$.
Let $K$ denote the (closed) set of condensation points of~$A$; it is also a 
member of~$M_0$.  
Since $K$ is compact metrizable and a member of~$M_0$, it follows that 
$\pr_{M_0}\restr K$ is one-to-one. 

We prove by contradiction that $A\setminus K$ is countable. 
If $A\setminus K$ is uncountable then we may choose for each~$\alpha$
a pair of points~$\orpr{x_\alpha}{y_\alpha}$ such that
$x_\alpha\in A\setminus K$, 
$y_\alpha\in K$ and $\pr_{M_\alpha}(x_\alpha) = \pr_{M_\alpha}(y_\alpha)$.

Let $J$ be any uncountable subset of~$\omega_1$
and let $y\in K$ be a condensation point of $\{x_\alpha:\alpha\in J\}$.
We show that $y$~also belongs to the closure of~$\{y_\alpha:\alpha\in J\}$.

Let $U$ be a basic open neighbourhood of~$y$ that is in~$M$ and take $\beta$ 
such that the support of~$U$ is contained in~$M_\beta$.
There are uncountably many~$\alpha$ in~$J\setminus\beta$ for 
which $x_\alpha\in U$ and for these~$\alpha$ we have 
$y_\alpha\restr\beta=x_\alpha\restr\beta$ and hence also $y_\alpha\in U$.
It follows that $y\restr M$ is in the closure of 
$\{y_\alpha\restr M:\alpha\in J\}$ and hence, because $\pr_M$ is one-to-one
on~$K$ that $y$~is in the closure of~$\{y_\alpha:\alpha\in J\}$.
\end{proof}

We are now ready to prove our result about weight $\omega_1$ fibered spaces. 
Remember from the introduction that $X$ is weight $\omega_1$ fibered if
there is a continuous map $f:X\to Y$ such that $Y$ and every fiber $f^{-1}(y)$
have weight at most~$\omega_1$.

\begin{theorem} 
Each csD space that is weight $\omega_1$ fibered is metrizable.
\end{theorem}

\begin{proof}
Let $X$ be a compact space that is weight $\omega_1$ fibered but not metrizable.
We will prove that $X$~is not csD.  
Let $\langle M_\alpha:\alpha\in \omega_1\rangle$ be an elementary sequence 
for~$X$.
As usual we are assuming that $X$ is a subspace of $[0,1]^\kappa$ for some
cardinal~$\kappa\in M_0$.  
Let $M$ denote the union $\bigcup_{\alpha\in \omega_1}M_\alpha$ and we
recall that $[x\restr M] $ denotes the set 
$\{y\in X:y\restr M=x\restr M\}$.
Since $X$ is weight~$\omega_1$ fibered, 
this will be \emph{witnessed} by the elementary submodel~$M_0$; thus it is 
routine to verify that, for each $x\in X$, the set $[x\restr M]$ has 
weight at most~$\omega_1$. 
By Hu\v{s}ek's result, we may assume that each set~$[x\restr M]$ is metrizable.
On the other hand, and also by Hu\v sek's theorem, we may also assume that the 
weight of~$X$ is greater than~$\omega_1$, so we may fix an $x\in X$ such that 
$[x\restr M]\neq \{x\}$.  
We complete the proof by establishing two lemmas of independent interest.  
If $[x\restr M]$ is not a $G_\delta$-set, then Lemma~\ref{2} will
complete the proof. 
On the other hand if $[x\restr M]$ is a $G_\delta$-set, then there will 
be some $\delta< \omega_1$ such that 
$[x\restr  M] = [x\restr M_\delta]$ is metrizable and 
not a singleton. 
In this case Lemma~\ref{3} will complete the proof.
\end{proof}

\begin{lemma}\label{2} 
If $X$ is a space for which there is an elementary chain 
$\langle M_\alpha:\alpha\in\omega_1\rangle$ for~$X$ and a point~$x$ in~$X$
such that $[x\restr M]$ is metrizable but not a $G_\delta$-set,
then $X$~is not~csD.
\end{lemma}

\begin{proof}
Since we are assuming that $[x\restr M]$ is not a $G_\delta$-set there 
is no $\delta\in \omega_1$ such that $[x\restr M]=[x \restr M_\delta]$.  
For each $\alpha\in \omega_1$ choose  
$x_\alpha \in [x\restr M_\alpha]\setminus [x\restr M]$.
For each $\alpha$, there is a $\beta_\alpha$ such that 
$x_\alpha\notin [x\restr M_{\beta_\alpha}]$, 
hence the set $A=\{x_\alpha:\alpha\in\omega_1\}$ is uncountable. 
The set of condensation points of~$A$ is contained in~$[x\restr M]$; 
and so, by Proposition~\ref{retract}, $X$~is not csD.
\end{proof}

\begin{lemma}\label{3} 
If for some $x\in X$, there is an elementary chain 
$\langle M_\alpha:\alpha\in\omega_1\rangle$ such that $[x\restr M]$ is 
metrizable but not a singleton, then $X$ is not csD.
\end{lemma}

\begin{proof}
Let $\langle M_\alpha:\alpha\in\omega_1\rangle$ be an elementary chain and 
suppose that $[x\restr M]$~is metrizable but not a singleton.
By Lemma~\ref{2}, we may assume that $[x\restr M]$ is a $G_\delta$-set.  
That is, we may assume that there is a $\delta\in\omega_1$ such that 
$[x\restr M]=[x\restr M_\delta]$.
So $[x\restr M_\delta]$ is metrizable and equal to $[x\restr M_\alpha]$
for all~$\alpha\ge\delta$.
We apply  elementarity to this statement to choose an elementary 
$\omega_1$-sequence that will not be $\omega_1$-separated.  
For $\alpha\geq \delta$, we make the observation that $M_{\alpha+1}$ is
a model of the statement
$$
(\exists x\in X) 
\bigl([x\restr M_\alpha ]=[x\restr M_\delta] 
  \text{ is  metrizable and not equal to } \{x\}\bigr).
$$
By elementarity, we can take such a point $x_\alpha \in X\cap M_{\alpha+1}$ and 
take $y_\alpha\in [x_\alpha \restr M_\alpha]\cap M_{\alpha+1}$ witnessing that
$[x_\alpha\restr M_\alpha]\neq \{ x_\alpha\}$.  
Since $[x_\alpha\restr M_\alpha]$ is a compact metrizable set which is a 
member of~$M_{\alpha+1}$, the basic open sets in~$M_{\alpha+1}$ will contain 
a base for it.

The remainder of the proof follows that of Gruenhage for the metrizably
fibered case (see~\cite{GaryDiag}), because the subspace
$\bigcup\bigl\{[x_\alpha\restr M_\delta]:
     \alpha\in\omega_1\setminus\delta\bigr\}$ 
is metrizably fibered over~$X_\delta$. 
Indeed, it follows from the construction that
$x_\alpha \restr M_\delta\neq x_\beta\restr M_\delta$ 
for $\delta<\alpha<\beta$.  
Let $J$ be any uncountable subset of~$\omega_1$.  
Working in the space $X_\delta$, there is an $\alpha\in J$ such that 
$\pr_{M_\delta}(x_\alpha)$ is a condensation point of the
set $\{\pr_{M_\delta}(x_\beta) : \beta\in J\}$. 
Since $\pr_{M_\delta}$ is a closed map and $[x_\alpha\restr M_\delta]$~is 
compact, there is a point $z\in [x_\alpha\restr M_\delta]$
that is a condensation point of $\{x_\beta:\beta \in J\setminus \alpha\}$. 
Although $z$ itself may not be a member of $M_{\alpha+1}$, the basic open sets
from $M_{\alpha+1}$ contain a local base at~$z$. 
In addition, for each basic open neighborhood~$U$ of~$z$ from~$M_{\alpha+1}$ 
and for each $\beta \in J\setminus\alpha+1$ we have 
$x_\beta\in U$ if and only if $y_\beta\in U$. 
It therefore follows that $z$ is also a condensation point of 
$\{ y_\beta:\beta\in J\}$. 
This shows that $\bigl<\orpr{x_\beta}{y_\beta}:\beta\in\omega_1\bigr>$ is 
not $\omega_1$-separated; and completes the proof that $X$~is not~csD.
\end{proof}

\subsection{Luzin sets}

A set of reals is a Luzin set if every uncountable subset is dense in some
interval.
Luzin sets exist if the Continuum Hypothesis holds but Martin's Axiom plus
the negation of Continuum Hypothesis implies they do not exist.
The existence of Luzin sets has some influence on the structure of csD spaces.

\begin{theorem} 
If there is a Luzin set, then every csD space contains points that 
are $G_\delta$-sets.
\end{theorem}

\begin{proof} 
Let $X$ be a csD space. 
If $X$ has any isolated points then there is nothing to prove, so assume that
it does not. 
Let $M$ be a countable elementary substructure of~$H(\theta)$ that 
contains~$X$.
Then $\pr_{M}[X]$ is a compact metrizable space with no isolated points.  
If there is any point $x\in X$ such that $[x\restr M] = \{x\}$ then we are
done as well.  
Let $Y$ be a dense subset of $\pr_M[X]$ that is homeomorphic
to the space of irrational numbers, and let $L\subset Y$ be a dense Luzin set 
of cardinality $\omega_1$.  
For each $z\in L$ choose distinct $x_z,y_z\in X$ so that 
$\pr_M(x_z)=\pr_M(y_z) = z$. 
We show that $\bigl<\orpr{x_z}{y_z}:z\in L\bigr>$ is not $\omega_1$-separated.  
Let $A$ be an uncountable subset of~$L$; we show that the closures
of $\{x_z : z\in A\}$ and $\{y_z : z\in A\}$ intersect.
 
Since $L$ is Luzin, there is a basic open set $W$ in $M$ such that $A$ 
contains a dense subset of $\pr_M[W]$. 
Since $X$ has a dense set of points of countable $\pi$-character we may choose
$x\in W\cap M$ so that it has a countable local $\pi$-base $\mathcal B$ 
consisting of basic open sets from~$M$ that are contained in~$W$.
However, since each member of~$\mathcal{B}$ belongs to~$M$ and $A\cap W$~is
dense in~$W$ each member of~$\mathcal{B}$ contains~$x_z$ and~$y_z$ for
some~$z$; this implies that $x$~is in the closure
of both $\{x_z : z\in A\}$ and $\{y_z : z\in A\}$.
\end{proof}

\subsection{Local $\pi$-bases and nets}

A family $\mathcal F$ of nonempty closed sets is a local $\pi$-net at a 
point~$x$ if every neighborhood of~$x$ contains a member of~$\mathcal F$. 
If $\mathcal F$ is a countable family of $G_\delta$-sets and is a local
$\pi$-net at $x$, then $x$ has a countable local $\pi$-base, provided
the ambient space is compact.

\begin{theorem}\label{t=w} 
Let $K$ be a compact $G_{\omega_1}$-set in a csD space $X$. 
Then each countable local $\pi$-base in $K$ expands to a countable local
$\pi$-net in $X$ consisting of $G_\delta$-sets.
\end{theorem}

\begin{proof} 
For ease of exposition we will assume that $X$ is zero-dimensional;
the modifications for the general case are tedious but straightforward.  
Let $x\in K$ and let $\{b_n:n\in \omega\}$ be a family of relatively clopen 
subset of~$K$ such that each neighborhood of~$x$ contains one.  
For each $n$ let $\mathcal K_n = \{ K^n_\alpha : \alpha \in \omega_1\}$ be
a filter base of clopen sets such that $b_n=\bigcap\mathcal{K}_n$.
Fix an ultrafilter $\mathcal U$ on $\omega$ so that for each  neighborhood~$U$ 
of~$x$, the set $\{ n : b_n \subset U\}$ is a  member of $\mathcal U$. 
Fix an elementary $\omega_1$-sequence 
$\langle M_\alpha:\alpha\in\omega_1\rangle$ for~$X$ so that 
$x$, $K$, $\{\mathcal K_n:n\in\omega\}$ and $\mathcal U$ are in $M_0$.  
For each $\alpha\in\omega_1$, 
assume that the family $\{ Z(\alpha,n):n\in \omega\}$ is not a local 
$\pi$-net at~$x$ where 
$Z(\alpha,n) =  \bigcap \{ K^n_\beta : \beta \in M_\alpha\}$.  
Let $U_\alpha\in  M_{\alpha+1}$ be a clopen set containing~$x$ so that
$Z(\alpha, n)\setminus U_\alpha\neq\emptyset$ for each $n\in\omega$.  
Choose $y_\alpha\in M_{\alpha+1}$ so that $y_\alpha$ is in the 
$\mathcal U$-limit of the sequence 
$\langle Z(\alpha,n)\setminus U_\alpha:n\in \omega\rangle$. 
It follows that $\pr_{M_\alpha}(y_\alpha) =  \pr_{M_\alpha}(x)$, 
and so $\bigl<\orpr{x}{y_\alpha}:\alpha\in\omega_1\bigr>$ is an elementary 
$\omega_1$-sequence and we prove it is not $\omega_1$-separated
by showing that $\langle y_\alpha:\alpha<\omega_1\rangle$ converges
co-countably to~$x$.

Let $W$ be any neighborhood of~$x$ and let $U = \{n\in\omega:b_n \subset W\}$.
For each $n\in U$ there is, by compactness some $\alpha_n<\omega_1$ so that
$K^n_{\alpha_n} \subset W$.  
Let $\alpha\in\omega_1$ be larger than all~$\alpha_n$. 
It follows that $Z(\delta, n) \subset W$ for all $n\in U$ and 
all $\delta\geq \alpha$.  
Since $y_\delta$ is in the closure of $\bigcup\{Z(\delta,n):n\in U\}$, 
we find that $\{y_\delta : \delta>\alpha\}\subset W$.  
\end{proof}

\begin{corollary} 
If $A$ is an uncountable subset of a csD space, then $A$ has a condensation 
point which has a countable local $\pi$-base in $\cl A$. 
\qed
\end{corollary}

\subsection{$\OCA$ and sequential compactness}

In the paper \cite{DoPa2} it is shown that the Proper Forcing Axiom
($\PFA$)
implies that all csD spaces are metrizable. 
We can use the results of this paper to give a shorter proof for sequentially 
compact csD spaces, and one that uses only a consequence of 
Todor\v{c}evi\'c's open coloring axiom ($\OCA$).  
$\PFA$~implies that compact spaces with countable tightness are sequential, 
but we do not know if $\OCA$ (or $\ZFC$!)\ implies that csD spaces are 
sequentially compact.

First we prove a strengthening of Gruenhage's result which shows how badly
non-metrizably fibered a csD non-metrizable space would have to be. 

\begin{lemma}\label{perfect} 
If $\{M_\alpha:\alpha\in\omega_1\}$ is an elementary chain for a 
non-metrizable csD space~$X$, then there is a $\delta\in\omega_1$ such that 
the set of non-metrizable sets in $\{[x\restr M_\delta ]\cap X_0:x\in X\}$ 
contains a perfect set.
\end{lemma}

\begin{proof}
It is implicit in~\cite{Husek} that a non-metrizable csD space 
(of countable tightness) will contain a separable non-metrizable subspace. 
By elementarity $M_0$~will contain such a separable set, 
and so $X_0= \cl(X\cap M_0)$ will itself not be metrizable.  
Fix any elementary $\omega_1$-sequence of pairs 
$\bigl<\orpr{x_\alpha}{y_\alpha}:\alpha\in\omega_1\bigr>$ 
for the sequence $\langle M_\alpha:\alpha\in\omega_1\rangle$, 
but chosen so that 
$\{x_\alpha, y_\alpha\}\subset X_0$ for all~$\alpha$.  
Let $A$ be an uncountable subset of~$\omega_1$ witnessing that the sequence 
is $\omega_1$-separated. 
Find a $\delta\in\omega_1$ so that for each basic open set~$U$ from~$M_\delta$
the implication 
``if $U\cap\{x_\alpha:\alpha\in A\}$ is uncountable, then 
     $U\cap\{x_\alpha:\alpha\in A\setminus M_\beta\}$ is
      infinite for all $\beta< \delta$'' holds. 
Let $K\subset X_{M_\delta}$ be the projection of 
$\cl{\{x_\alpha:\alpha\in A\setminus M_\delta\}}$ by the map~$\pr_{M_\delta}$.  
It follows from the choice of~$\delta$ that $K$~is a perfect set. 
In addition each point in~$K$ is a limit point of 
$\pr_{M_\delta} \bigl[\{ x_\alpha:\alpha\in A\setminus M_\gamma\}\bigr]$ 
for all $\gamma\in \omega_1$. 

We show that for each~$z$ in~$K$ the set
$[z]=\pr_{M_\delta}^{-1}(z)$ is not metrizable.
We assume we have a~$z$ in~$K$ such that $[z]$~is metrizable and derive
a contradiction.
For each $x\in [z]$ the set $[x\restr M]$ is metrizable and a $G_\delta$-set
because it is a subset of~$[z]$. 
By Lemma~\ref{3} it follows that $[x\restr M] = \{x\}$.  
Therefore, the mapping $\pr_{M}$ restricted to $[z]$ is a homeomorphism. 
Since $\pr_M\bigl[[z]\bigr]$~is a compact metrizable subset 
of~$[0,1]^{\kappa\cap M}$
there is some $\alpha_z<\omega_1$ such that the basic open sets in~$M_{\alpha_z}$
contain a base for~$[z]$.  
Now choose a sequence $\langle\beta_n : n\in\omega\rangle$ in
$A\setminus \alpha_z$ so that 
$\langle\pr_{M_\delta}(x_{\beta_n}) : n\in \omega\rangle$ converges to~$z$. 
There is a point $x\in [z]$ that is a cluster point of the sequence 
$\langle x_{\beta_n} : n\in\omega\rangle$. 
By thinning out we can assume that the latter sequence converges to~$x$.  
Since the sequence $\langle\pr_{M_{\alpha_z}}(x_{\beta_n}):n\in \omega\rangle$ 
converges to~$\pr_{M_{\alpha_z}}(x)$ in~$X_{\alpha_z}$, the sequence
$\langle \pr_{M_{\alpha_z}}(y_{\beta_n}):n\in\omega\rangle$ converges 
to~$\pr_{M_{\alpha_z}}(x)$ as well.  
Since $\langle y_{\beta_n}:n\in\omega\rangle$ accumulates at some point 
in~$[z]$ and $\pr_{M_{\alpha_z}}$ is one-to-one on~$[z]$, 
it follows that $x$~is a limit of $\langle y_{\beta_n} : n\in \omega\rangle$.
\end{proof}

We will make use of the following application of $\OCA$ by Todor\v{c}evi\'c. 
Let $\mathcal X$ be a family of disjoint pairs of subsets of a countable 
set~$S$.  
Say that the family~$\mathcal X$ is \emph{countably separated} if there is a 
countable family~$\mathcal Y$ of subsets of~$S$ such that for each 
pair $\orpr ab\in \mathcal X$, there is a $Y\in \mathcal Y$ such that 
$a\setminus Y$ and $b\cap Y$ are both finite --- in case $\mathcal{Y}$
has just one element we say that $\mathcal{X}$~is \emph{separated}.
The following result is taken from \cite{Farah00}*{p.~145} and
is attributed to Todor\v{c}evi\'c.

\begin{proposition}[OCA] \label{oca}
If a family $\mathcal X$ of disjoint pairs of subsets of a countable set~$S$ 
is not countably separated, then there is an uncountable subcollection 
$\{ \orpr{a_\alpha}{b_\alpha}:\alpha\in\omega_1\}$ of~$\mathcal X$ with 
the property that whenever $\alpha\neq\beta$ in~$\omega_1$
the set $(a_\alpha\cap b_\beta)\cup (a_\beta\cap b_\alpha)$ is not empty. 
In particular, for all uncountable $A\subset\omega_1$ the collection
$\{\orpr{a_\alpha}{b_\alpha}:\alpha \in A\}$ is not separated.
\end{proposition}

\begin{theorem}[OCA] 
If $X$ is a sequentially compact csD space then $X$ is metrizable. 
\end{theorem}

\begin{proof} 
Assume that $X$ is not metrizable, and apply Lemma \ref{perfect}.  
Let $S$ denote the countable set $M_0\cap X$. 
We work in the non-metrizable subspace $X_0$. 
Let $f$ denote the projection map from~$X_0$ onto~$X_{M_0}$,
which is onto by elementarity.
Let $Z$ be a perfect set of points of~$X_0$ with the property that 
$f^{-1}(z) = [z]$ is not metrizable for each $z\in Z$.
Let $\{\mathcal{Y}_z:z\in Z\}$ be a listing of all countable sequences of 
subsets of~$S$. 
For each $z\in Z$ we will show that there are disjoint subsets $a_z$ and~$b_z$
of~$S$ that are not separated by~$\mathcal Y_z$ and converge to distinct points
of~$[z]$.

Let $\{Y_n : n\in \omega\}$ be an enumeration of~$\mathcal Y_z$.
For each $n\in \omega$, let $Y^0_n = S\setminus Y_n$ and  $Y^1_n = Y_n$. 
For each function $h\in 2^\omega$, let $[z]_h$ denote
the closed set $\bigcap_{n\in\omega}\cl\bigcap_{i<n}Y^{h(i)}_i$. 
Since $[z]$ is not metrizable, there must be some  $h\in 2^\omega$ such that 
$[z]_h$ is not a singleton. 
Choose open sets $U$ and $W$ of $X_0$, with disjoint closures, that
both intersect~$[z]_h$.
Additionally, fix a descending neighborhood base $\{U_n:n\in \omega\}$ for the 
$G_\delta$-set $[z]$. 
Each of the families $\{U\cap U_n\cap\bigcap_{i<n}Y^{h(i)}_i:n\in \omega\}$
and $\{W\cap U_n\cap\bigcap_{i<n}Y^{h(i)}_i:n\in \omega\}$ are descending 
sequences of infinite subsets of~$S$. 
There are infinite sets $a_z$ and~$b_z$ such that for each $n$,
$a_z$ is almost contained in $U\cap U_n\cap \bigcap_{i<n} Y^{h(i)}_i$
and $b_z$ is almost contained in $W\cap U_n\cap \bigcap_{i<n}Y^{h(i)}_i$. 
Since we are assuming that $X$ is sequentially compact,  we may assume that 
$a_z$ converges to a point $x_z\in[z]\cap\cl{U}$ and 
$b_z$ converges to a point $y_z\in[z]\cap\cl{W}$. 

Now we apply Lemma \ref{oca} to the family 
$\mathcal X =\{ \orpr{a_z}{b_z}:z\in Z\}$. 
It is evident by the construction that this family is not countably separated.
Therefore there is an uncountable subset~$Y$ of $Z$ such that for all 
uncountable $A\subset Y$ the families $\{ a_z: z\in A\}$ and $\{b_z:z\in A\}$
can not be separated.

It now follows easily that the $\omega_1$-sequence of pairs 
$\bigl<\orpr{x_z}{y_z}:z\in Y\bigr>$  is not $\omega_1$-separated. 
To see this assume that $U$ is an open set containing 
$\{x_z :z\in A\}$ for some uncountable~$A$.
Then $a_z\setminus U$ will be finite for each $z\in A$. 
By considering all possible uncountable subsets of~$A$ it follows that
for all but countably many~$z$ in~$A$ the intersection $b_z\cap U$
is infinite and hence $y_z\in\cl U$ for all these~$z$.
\end{proof}

\section{Questions}

Needless to say, the main open problem is to determine if every  csD space is 
metrizable. 
However, here are some other questions which certainly seem difficult and 
interesting.

\begin{question} Assume that $X$ is an infinite  csD space.
\begin{enumerate}
\item Is $X$ sequentially compact? 
      Does it even contain a non-trivial converging sequence? 
      Does $X$ contain a point of countable character? 
\item If $X$ is first countable, is it metrizable?
\item Does $X$ contain a copy of the Cantor set?
\item Can the cardinality of $X$ be greater than  $2^{\aleph_0}$?
\item Is each countably compact subset closed?
\item Does Martin's Axiom imply that $X$ is metrizable?
\item If $X$ is hereditarily separable, is it metrizable? Gruenhage
  \cite{GaryDiag}   has shown that hereditarily Lindel\"of csD spaces
  are metrizable.
\end{enumerate}
\end{question}

Another question is what happens if we modify the csD property by analogy
with separated versus countably separated families of pairs.

\begin{question}
Define $X$ to be $\sigma$-sD to mean that for each collection 
$\{ \orpr{x_\alpha}{y_\alpha} : \alpha\in \omega_1\}$ of pairs from~$X$, 
there is a countable cover $\{A_n : n\in \omega\}$  of $\omega_1$, such that 
for each $n$ the sets $\{ x_\alpha : \alpha\in   A_n\}$ 
and $\{ y_\alpha : \alpha \in A_n\}$ have disjoint closures. 
If $X$ is $\sigma$-sD, is it metrizable? 
\end{question}

\end{document}